\documentclass[12pt,reqno]{amsart}
\usepackage{amsmath,amssymb,amsthm,amsfonts,mathtools,enumerate}
\usepackage[alphabetic]{amsrefs}
\usepackage{datetime,graphicx}
\usepackage[usenames,dvipsnames]{color}
\usepackage{hyperref}
\hypersetup{pdfcreationdate=\pdfdate,pdfmoddate=\pdfdate,colorlinks=true,citecolor=BlueViolet,linkcolor=BlueViolet,urlcolor=BlueViolet}
\usepackage[margin=1in]{geometry}

\numberwithin{equation}{section}

\newtheorem{thm}[equation]{Theorem}
\newtheorem*{thmnonum}{Theorem}
\newtheorem{prop}[equation]{Proposition}
\newtheorem{lemma}[equation]{Lemma}
\newtheorem{cor}[equation]{Corollary}

\theoremstyle{definition}

\newtheorem{eg}[equation]{Example}

\newcommand{\C}{{\mathbb{C}}}
\newcommand{\R}{{\mathbb{R}}}
\newcommand{\Z}{{\mathbb{Z}}}
\newcommand{\N}{{\mathbb{N}}}

\newcommand{\A}{\mathbb{A}}
\newcommand{\D}{\mathbb{D}}

\newcommand{\abs}[1]{\left\lvert#1\right\rvert}
\newcommand{\norm}[1]{\left\lVert {#1} \right\rVert}
\newcommand{\ip}[2]{\left\langle #1, #2 \right\rangle}
\newcommand{\B}{\mathcal{B}}
\newcommand{\BP}{\mathcal{B}_{\Omega}}
\newcommand{\barf}{\bar{f}}

\author[R. Sivaguru]{Sivaguru Ravisankar}
\address{School of Mathematics, Tata Institute of Fundamental Research,
Mumbai 400005, India}
\email{sivaguru@math.tifr.res.in}

\author{Yunus E. Zeytuncu}
\address{Department of Mathematics and Statistics, University of Michigan-Dearborn,
Dearborn, Michigan 48128, USA}
\email{zeytuncu@umich.edu}
\thanks{The work of the second author is partially supported by a grant from the Simons Foundation (\#353525).}

\date{August 25, 2016}

\title[Smoothing properties of the Bergman projection]{A note on smoothing properties of the Bergman projection}

\begin{document}

\maketitle

\begin{abstract}
Recently Herbig, McNeal, and Straube have showed that the Bergman projection of conjugate holomorphic functions is smooth up to the boundary on smoothly bounded domains that satisfy condition R.
We show that a further smoothing property holds on a family of Reinhardt domains; namely, the Bergman projection of conjugate holomorphic functions is holomorphic past the boundary.
\end{abstract}


\section{Introduction}\label{sec:Intro}

Let $\Omega$ be a bounded domain in $\mathbb{C}^n$. The Bergman projection $\BP$ is the orthogonal projection operator from the space of square integrable functions $L^2(\Omega)$ onto the closed subspace of square integrable holomorphic functions $A^2(\Omega)$. The boundedness of $\BP$ on other function spaces is of considerable interest and it has deep connections with the geometry of the domain. For example, if $\Omega$ is strongly pseudoconvex then the projection operator is bounded on all $L^p$-Sobolev spaces $W^{p,k}(\Omega)$ for all $p\in(1,\infty)$ and $k\geq 0$ \cite{PhoSte}.
On the other hand, if the operator $\BP$ is bounded on all $L^2$-Sobolev spaces $W^{k}(\Omega)$ for $k\geq 0$ then any biholomoprhic map from a smooth $\Omega$ to another smoothly bounded pseudoconvex domain extends smoothly to the closure of domains \cite{Bell}.

The Bergman projection preserves holomorphic functions in $L^2(\Omega)$ and hence the holomorphic $L^2$-Sobolev space at each scale. Therefore, one can not expect any gain in the Sobolev scale under $\BP$.
However, in two recent papers \cite{HerMcN} and \cite{HerMcNStr} the authors have shown that $\BP$ gains derivatives in certain directions on smoothly bounded domains that satisfy condition R.
A domain $\Omega$ satisfies condition R if $\BP$ maps the space of functions smooth up to the boundary $C^{\infty}(\overline{\Omega})$ to itself.
Out of a few different formulations of this gain, we highlight the following version.

\begin{thmnonum}\cite[Corollary 1.12]{HerMcNStr}\label{hms}
Let $\Omega$ be a smoothly bounded domain in $\mathbb{C}^n$ that satisfies condition R. Then for any $f\in A^2(\Omega)$, $\BP\barf$ is smooth up to the boundary.
\end{thmnonum}

We note that even when the function $f$ does not belong to any Sobolev space $W^k(\Omega)$ for $k>0$, the projection $\BP\barf$ automatically lands in all $W^k(\Omega)$. In other words, on such domains the Bergman projection $\BP$ \textit{smoothens} conjugate holomorphic functions.

Moreover, the following observation on complete Reinhardt domains suggests a further smoothing may hold.
On such a domain $\Omega$, $\BP \bar{f}$ is a constant for $f\in A^2(\Omega)$.
This follows from $\{z^{\alpha}/c_\alpha\,:\,\alpha\in (\N\cup\{0\})^n\}$ forming an orthonormal basis for $A^2(\Omega)$.
Consequently, not only is $\BP\barf$ smooth up to the boundary, it is, in fact, holomorphic on a larger domain (actually holomorphic on $\mathbb{C}^n$). This example suggests that $\BP$ may have a further smoothing property and we indeed prove such a result on a family of Reinhardt domains.

\begin{thm}\label{thm:Main}
Let $\Omega\subset\mathbb{C}^n$ be a bounded Reinhardt domain with $C^1$ boundary such that $\overline{\Omega}$ does not intersect the coordinate axes. Then, for every $f\in A^2(\Omega)$, $\B_\Omega\bar{f}$ extends holomorphically to a (strictly) larger domain.
\end{thm}

In contrast to the result of Herbig, McNeal, and Straube, the above theorem describes a new phenomenon on a different class of domains.
Even on smooth Reinhardt domains which satisfy condition R (see \cite{Boas84, Straube86}), our result draws a stronger conclusion.
Note that on any smoothly bounded pseudoconvex domain there exists a holomorphic function that is smooth up to the boundary and does not extend past any boundary point \cite{Catlin80}. 
Furthermore, it is not known whether condition R holds on the domains we consider without smoothness of the boundary.

The Bergman projection acting on conjugate holomorphic functions is also known as the Friedrichs operator.
The study of this operator has a long history, see \cite{Friedrichs,KrantzEtAl}, with deep connections to the study of quadrature domains in complex analysis.
In a subsequent article, we plan to explore the implications of our results for the Friedrichs operator and quadrature domains.

Our main tool in proving Theorem \ref{thm:Main} is observing it in the special case of a product of annuli in $\mathbb{C}^n$. We show this in Proposition \ref{prop:ExtnProdAnn} and this can be rephrased as the following local smoothing property. Once we prove the local statement, the global version follows from a patching argument.

\begin{thm}\label{thm:MainLocal}
Let $\Omega\subset\mathbb{C}^n$ be a bounded Reinhardt domain and $P\subset\Omega $ be a product of annuli such that $\overline{P}$ does not intersect the coordinate axes. Then, for every $f\in A^2(\Omega)$, $\B_\Omega\bar{f}$ extends holomorphically past $bP\cap b\Omega$.
\end{thm}

In the next section, we go over some basic facts about the Bergman projection and some examples that lead into the main theorem. We prove the main statements in Section \ref{sec:MainThmPf}. We conclude with some further remarks in the last section.


\section{Setup and Examples}\label{sec:BergProj}

\subsection{Basic Setup} In this section, we focus on bounded Reinhardt domains. Let $\Omega$ be a bounded Reinhardt domain in $\C^n$; that is, if $(z_1,\ldots,z_n)\in\Omega$ and $\theta_1,\ldots,\theta_n\in\R$, then $\big(e^{i\theta_1}z_1,\ldots,e^{i\theta_n}z_n)\in\Omega$.
We denote the Hilbert space of square integrable functions with respect to the Lebesgue measure $dV$ by $L^2(\Omega)$.
The subspace of holomorphic functions in $L^2(\Omega)$ is called the Bergman space and it is denoted by $A^2(\Omega)$. It follows from the mean value property of holomorphic functions that $A^2(\Omega)$ is a closed subspace of $L^2(\Omega)$. Hence, there exists an orthogonal projection operator from $L^2(\Omega)$ onto $A^2(\Omega)$ and it is called the Bergman projection $\BP$.

The set (or a subset) of monomials $\{z^{\alpha}:\alpha\in\Z^n\}$ forms an orthogonal basis for $A^2(\Omega)$.
Let $c_{\alpha}^2=\int_{\Omega}|z^{\alpha}|^2dV(z)$.
Then, the set $\left\{z^{\alpha}/c_{\alpha}\right\}$ forms an orthonormal basis for the Bergman space.
It turns out that $A^2(\Omega)$ is a reproducing kernel Hilbert space and we have the following representation for the Bergman kernel by using this orthonormal basis,  
$$B_{\Omega}(z,w)=\sum_{\alpha\in\Z^n}\frac{z^{\alpha}\bar{w}^{\alpha}}{c_{\alpha}^2}$$
and the Bergman projection is the integral operator associated to this kernel,
$$\BP f=\int_{\Omega}B_{\Omega}(z,w)f(w)dV(w)$$
for any $f\in L^2(\Omega)$. We refer to \cite{KrantzBook} for more on the basics of the Bergman projection.

Originally defined only on $L^2(\Omega)$, it is known that $\BP$ is a bounded operator on other function spaces under various geometric conditions. We refer to \cite{BoasStraubeSurvey} for a survey of results on Sobolev spaces and to \cite{McNSte94, ZeytuncuTran} and the references therein for a sample of results on $L^p$ spaces. 


\subsection{Examples}\label{sec:Eg}

We begin by studying a motivating example - an annulus in $\C$.
An elementary calculation yields that the Bergman projection of a conjugate holomorphic function in $L^2$ extends holomorphically to a larger domain.
For $0\le r<R$, let $\A(r,R)$ denote the annulus $\{r<\abs{z}<R\}$.
\begin{lemma}
Let $f\in A^2\left(\A(r,1)\right)$ for $0<r<1$. Then, $\B_{\A(r,1)}\bar{f}$ is holomorphic extends holomorphically to $\A\left(r^2,\dfrac{1}{r}\right)$.
\end{lemma}
\begin{proof}
For $j\in\Z$, let $c_j=\norm{z^j}_{L^2(\A(r,1))}$ so that $\{z^j/c_j\}$ is an orthonormal basis for $A^2(\A(r,1))$.
Let us use the Laurent expansion $f=\sum_j\, f_jz^j$ of $f\in A^2(\A(r,1))$ to compute $\B_{\A(r,1)}\bar{f}$.
\[\B_{\A(r,1)}\bar{f} = c_0^2\sum\limits_{j=-\infty}^\infty \frac{\overline{f_{-j}}}{c_{j}^2} z^j\]
since, for $j\in\Z$,
\[\ip{\bar{f}}{\frac{z^j}{c_j}}
=\sum\limits_{k=-\infty}^\infty\overline{f_k} \int\limits_{\A(r,1)}\,\bar{z}^k\frac{\bar{z}^j}{c_j}\, dA(w)
=\frac{c_0^2}{c_j}\overline{f_{-j}}.\]
Calculating the $c_j$'s and organizing the terms suitably yields
\[
\B_{\A(r,1)}\bar{f}(z)
= \left(\frac{1-r^2}{-2\ln r}\right)\frac{\overline{f_1}}{z} + \left(\frac{1-r^2}{r^2}\right)\sum\limits_{j=2}^{\infty}\frac{j-1}{1-r^{2(j-1)}}\cdot \frac{\overline{f_j}r^{2j}}{z^j} + (1-r^2)\sum\limits_{j=0}^{\infty}\frac{j+1}{1-r^{2(j+1)}}\cdot \overline{f_{-j}}z^j.
\]
The conclusion follows by noting that the middle term converges (uniformly absolutely on compact subsets of) on $\abs{r^2/z} < 1$ and the last term on $\abs{1/z} > r$.
\end{proof}
A simple scaling argument gives us the following generalization.
\begin{cor}\label{cor:GenAnnExtn}
Let $f\in A^2\left(\A(r,R)\right)$ for $0<r<R$. Then, $\B_{\A(r,R)}\bar{f}$ extends holomorphically to $\A\left(\dfrac{r^2}{R},\dfrac{R^2}{r}\right)$.
\end{cor}

If we consider an annulus with inner radius $0$, that is, the punctured disc, we are in a rather simple situation where the Bergman projection of an $L^2$ conjugate holomorphic function is a constant; in particular, an entire function.

\begin{eg}\label{eg:PuncDisc}
Consider the punctured disc $\D\setminus\{0\}\subset\C$ and let $f\in A^2(\D\setminus\{0\})$.
Since $f$ has a removable singularity at $0$ it extends to a function in $A^2(\D)$ which we denote by $f$ as well.
In particular, $A^2(\D\setminus\{0\})$ and $A^2(\D)$ are spanned by $\{z^j:j\ge 0\}$.
Since, the $L^2$ norms on $D$ and $D\setminus\{0\}$ agree, the Bergman projections on $\D\setminus\{0\}$ and $\D$ agree and we have
\[\B_{\D\setminus\{0\}}\bar{f} = \B_{\D}\bar{f} = \overline{f(0)},\]
a constant.
\end{eg}

The following example of the Hartogs triangle, in which the origin is in the closure of the domain, however, presents a much subtler picture.
The global version of the main theorem fails to hold whereas the local version holds everywhere except at the origin.
This highlights the importance of the assumption that $\overline{\Omega}$ not intersect the coordinate axes in Theorem \ref{thm:Main}.

\begin{eg}[Hartogs triangle]\label{eg:HartogsTri}
Consider the Hartogs triangle $\Omega=\{\abs{z}< \abs{w} < 1\}\subset\C^2$. Then, $A^2(\Omega)$ is spanned by $\{z^jw^k\,:\,j\ge 0, j+k+1\ge 0\}$. Furthermore, note that for $j,\alpha \ge 0$, $j+k+1\ge 0$, and $\alpha+\beta+1 \ge 0$,
\[\langle\bar{z}^\alpha\bar{w}^\beta,z^j w^k \rangle = \int\limits_\Omega\, \bar{z}^\alpha\bar{w}^\beta\bar{z}^j\bar{w}^k\, dV = \begin{cases}
\operatorname{vol}(\Omega) &\text{if } \alpha=j=0 \text{ and } \beta=-k,\\
0 & \text{otherwise}.
\end{cases}\]
Therefore, for $ f=\sum\limits_{j\ge 0,\,  j+k+1\ge 0}f_{j,k}z^jw^k\in A^2(\Omega)$,
\[
\B_\Omega\bar{f}
= \sum\limits_{\substack{j\ge 0,\\  j+k+1\ge 0}}\ \sum\limits_{\substack{\alpha\ge 0,\\ \alpha+\beta+1\ge 0}}\overline{f_{\alpha,\beta}}\ip{\bar{z}^\alpha\bar{w}^\beta}{z^jw^k}\frac{z^jw^k}{\norm{z^jw^k}^2}
= \frac{a}{w} + b+ cw
\]
for some constants $a$, $b$, and $c$.

Notice that there is a larger domain $\Omega'$ to which the functions $\B_\Omega\bar{f}$ extend holomorphically but $\overline{\Omega}\not\subset\Omega'$.

By considering the {\it fat} Hartogs triangle (see \cite{EdholmMcNeal1,EdholmMcNeal2}) $\Omega=\{\abs{z}^\gamma< \abs{w} < 1\}$, for $\gamma>0$, we can arrange $\B_\Omega\bar{f}$ to have (up to) any prescribed finite order of blow-up at $0$.
\end{eg}


\section{Main Theorem and its Proof}\label{sec:MainThmPf}

We begin with a proposition that proves the main theorem for a product of annuli.
We will use this result to prove the local version of the main theorem which in turn is used to prove the main theorem.

\begin{prop}\label{prop:ExtnProdAnn}
Let $P\subset\C^n$ be a product of annuli such that $\overline{P}$ does not intersect the coordinate axes.

Then, there exists a product of annuli $P'\subset\C^n$ such that 
\begin{enumerate}[(i)]\itemsep\medskipamount
\item $\overline{P}\subset P'$,
\item $\overline{P'}$ does not intersect the coordinate axes, and
\item for every $f\in A^2(P)$,  $\B_P\bar{f}\in A^2(P')$. i.e., there exists $F\in A^2(P')$ such that $\left.F\right\vert_P = \B_P\bar{f}$.
\end{enumerate}
\end{prop}
\begin{proof}
Let $P'\subset\C^n$ be any product of annuli such that $\overline{P}\subset P'$ and $\overline{P'}$ does not intersect the coordinate axes.
We will place further restrictions on $P'$ to ensure that it satisfies the other conclusions in the statement.

For $\alpha\in\Z^n$, let $d_\alpha=\norm{z^{\alpha}}_{L^2(P)}$ and $e_\alpha=\norm{z^{\alpha}}_{L^2(P')}$ so that $\{z^{\alpha}/d_\alpha\}$ and $\{z^{\alpha}/e_\alpha\}$ are orthonormal bases for $A^2(P)$ and $A^2(P')$ respectively.
Let $f=\sum_{\alpha}f_\alpha z^\alpha\in A^2(P)$.
Then,
\[\B_P\bar{f}=\sum\limits_{\alpha\in\Z^n}\ip{\bar{f}}{\frac{z^\alpha}{d_\alpha}}_{L^2(P)}\frac{z^\alpha}{d_\alpha}
=\sum\limits_{\alpha,\beta\in\Z^n}\overline{f_\beta}\ip{\bar{z}^\beta}{\frac{z^\alpha}{d_\alpha}}_{L^2(P)}\frac{z^\alpha}{d_\alpha}
=\sum\limits_{\alpha\in\Z^n}\overline{f_{-\alpha}}\left(\frac{d_0^2}{d_\alpha}\right)\frac{z^\alpha}{d_\alpha}\]
and hence
\begin{align*}
\norm{\B_P\bar{f}}_{L^2(P')}^2&=\sum\limits_{\alpha\in\Z^n}\abs{\ip{\B_P\bar{f}}{\frac{z^\alpha}{e_\alpha}}_{L^2(P')}}^2
=\sum\limits_{\alpha\in\Z^n}\abs{\sum\limits_{\beta\in\Z^n}\overline{f_{-\beta}}\left(\frac{d_0^2}{d_\beta^2}\right)\ip{z^\beta}{\frac{z^\alpha}{e_\alpha}}_{L^2(P')}}^2\\
&=\sum\limits_{\alpha\in\Z^n}\abs{\overline{f_{-\alpha}}\left(\frac{d_0^2}{d_\alpha^2}\right)e_\alpha}^2
=d_0^4\sum\limits_{\alpha\in\Z^n}\abs{f_{-\alpha}}^2d_{-\alpha}^2\left(\frac{e_\alpha}{d_{-\alpha}d_\alpha^2}\right)^2.
\end{align*}
Since $\sum_\alpha\abs{f_{-\alpha}}^2d_{-\alpha}^2=\norm{f}_{L^2(P)}^2 < \infty$, it suffices to choose $P'$ so that $e_\alpha\le d_{-\alpha}d_{\alpha}^2$ for all $\alpha\in\Z^n$.
That is, we need to find $P'$ so that the following holds: \begin{equation}\label{eqn:CompEstAnn}
\int\limits_{P'} \abs{z^\alpha}^{2}\,dV(z) \le C \left(\int\limits_{P} \abs{z^\alpha}^2\,dV(z)\right)^2\int\limits_{P} \abs{z^{-\alpha}}^2\,dV(z)
\end{equation}
for all $\alpha\in\Z^n$ and some $C>0$ which is independent of $\alpha$.

Since $P$ and $P'$ are products of annuli, we will be done if we show the following: given an annulus $\A\subset\C$ there exists an annulus $\A'\subset\C$ such that $\overline{\A}\subset \A'$, $0\notin\overline{\A'}$, and the pair $\A$ and $\A'$ satisfy the above inequality.
Suppose $\A=\A(r,R)$ for some $0<r<R$.
Then, one can check that $\A'=\A(r',R')$ satisfies the above criteria for any $r'$ and $R'$ satisfying
\[\frac{r^2}{R}<r'<R'<\frac{R^2}{r}.\]
This is in agreement with the conclusion of Corollary \ref{cor:GenAnnExtn} which guaranteed holomorphic extension of $\B_{\A}\bar{f}$, for $f\in A^2(\A)$, to $\A(r^2/R,R^2/r)$ from which it follows that the extension is $L^2$ in any compact subset of $\A(r^2/R,R^2/r)$. i.e., $\B_{\A}\bar{f}\in A^2(\A(r',R'))$.
\end{proof}

Note that the closure of $P$ not intersecting the coordinate axes was crucial in the above proof.
This was highlighted earlier in Example \ref{eg:HartogsTri} as well.
Let us now proceed to prove the local version of the main theorem.

\begin{thm}[Local Version]\label{thm:MainLocalFull}
Let $\Omega\subset\mathbb{C}^n$ be a bounded Reinhardt domain and $P\subset\Omega $ be a product of annuli such that $\overline{P}$ does not intersect the coordinate axes.

Then, for every $f\in A^2(\Omega)$, $\B_\Omega\bar{f}$ extends holomorphically past $bP\cap b\Omega$.

More precisely, there exists a product of annuli $P'$ such that
\begin{enumerate}[(i)]\itemsep\medskipamount
\item $\overline{P}\subset P'$, and hence $bP\cap b\Omega\subset P'$,
\item $\overline{P'}$ does not intersect the coordinate axes, and
\item for every $f\in A^2(\Omega)$,  $\B_\Omega\bar{f}\in A^2(P')$. i.e., there exists $F\in A^2(P')$ such that $\left.F\right\vert_P = \left.\B_\Omega\bar{f}\right\vert_P$.
\end{enumerate}
\end{thm}
\begin{proof}
Write $f\in A^2(\Omega)$ as $f=\sum_{\alpha}f_\alpha z^\alpha$ with $f_\alpha=0$ when $z^\alpha\notin A^2(\Omega)$. For $\alpha\in\Z^n$, let $c _\alpha = \norm{z^\alpha}_{L^2(\Omega)}$ and $d _\alpha = \norm{z^\alpha}_{L^2(P)}$ noting that some of the $c_\alpha$'s might be infinite.
Hence,
\begin{align*}
\B_\Omega\bar{f} &= \sum\limits_{\alpha\in \Z^n}\,\left\langle \sum\limits_{\beta\in\Z^n}\,\overline{f_\beta} \bar{z}^\beta,\frac{z^\alpha}{c_\alpha}\right\rangle\cdot\frac{z^\alpha}{c_\alpha}  = c^2_0\sum\limits_{\alpha\in\Z^n}\,\overline{f_{-\alpha}}\cdot\frac{z^\alpha}{c^2_\alpha}, \text{ and }\\[3mm]
\B_P\bar{f} &= \sum\limits_{\alpha\in \Z^n}\,\left\langle \sum\limits_{\beta\in\Z^n}\,\overline{f_\beta} \bar{z}^\beta,\frac{z^\alpha}{d_\alpha}\right\rangle\cdot\frac{z^\alpha}{d_\alpha}  = d^2_0\sum\limits_{\alpha\in\Z^n}\,\overline{f_{-\alpha}}\cdot\frac{z^\alpha}{d^2_\alpha}.
\end{align*}
By Proposition \ref{prop:ExtnProdAnn} there exists a larger product of annuli $P'$ such that $\overline{P}\subset P'$ and $\B_P\bar{f}\in A^2(P')$.
If we let $e_\alpha=\norm{z^\alpha}_{L^2(P')}$ for $\alpha\in\Z^n$, we have,
\[
\norm{\B_\Omega\bar{f}}^2_{L^2(P')} = c_0^4\sum\limits_{\alpha\in\Z^n}\abs{f_{-\alpha}}^2\left(\frac{e_\alpha^2}{c_\alpha^4}\right)
\quad\text{and}\quad
\norm{\B_P\bar{f}}^2_{L^2(P')} = d_0^4\sum\limits_{\alpha\in\Z^n}\abs{f_{-\alpha}}^2\left(\frac{e_\alpha^2}{d_\alpha^4}\right) .
\]
We conclude that  $\B_\Omega\bar{f}\in A^2(P')$ since $\norm{\B_P\bar{f}}^2_{L^2(P')}<\infty$ and $d_\alpha \le c_\alpha$ for every $\alpha\in\Z^n$.
\end{proof}

\begin{thm}[Main Theorem]\label{thm:MainFull}
Let $\Omega\subset\mathbb{C}^n$ be a bounded Reinhardt domain with $C^1$ boundary such that $\overline{\Omega}$ does not intersect the coordinate axes.

Then, for every $f\in A^2(\Omega)$, $\B_\Omega\bar{f}$ extends holomorphically to a (strictly) larger domain.

More precisely, there exists a log-convex Reinhardt domain $\Omega'\subset\C^n$ such that 
\begin{enumerate}[(i)]\itemsep\medskipamount
\item $\overline{\Omega}\subset\Omega'$, and
\item for every $f\in A^2(\Omega)$,  $\B_\Omega\bar{f}$ extends holomorphically to $\Omega'$.
\end{enumerate}
\end{thm}
\begin{proof}
We prove, in the next paragraph, that
\begin{equation}\label{eqn:ChooseP_z}
\forall\, z\in b\Omega, \exists \text{ a product of annuli } P_z\subset \Omega, z\in P_z'
\end{equation}
where $P_z'$ is the product of annuli corresponding to $P_z$ from Theorem~\ref{thm:MainLocalFull}.
With this in hand the proof is completed as follows. Since $\{P'_z: z\in b\Omega\}$ covers $b\Omega$, there exists $z_1,\ldots, z_N \in b\Omega$ such that $\{P'_{z_j}: 1\le j\le N\}$ covers $b\Omega$.
In fact, this sub-collection covers a full neighbourhood of $b\Omega$.
Now, let 
\[\Omega' = \Omega \cup P'_{z_1} \cup \cdots \cup P'_{z_N}.\]
Clearly $\overline{\Omega}\subset \Omega'$ and $\Omega'$ is a Reinhardt domain.
For $f\in A^2(\Omega)$, $\B_\Omega\bar{f}\in A^2(P'_{z_j})$, for $1\le j\le N$, by Theorem~\ref{thm:MainLocalFull}.
Hence, $\B_\Omega\bar{f}\in A^2(\Omega')$ and, in particular, it is holomorphic in $\Omega'$.
Furhtermore, we may replace $\Omega'$ by its pseudoconvex hull which is the smallest log-convex Reinhardt domain containing it.

\begin{figure}[h!]
\includegraphics{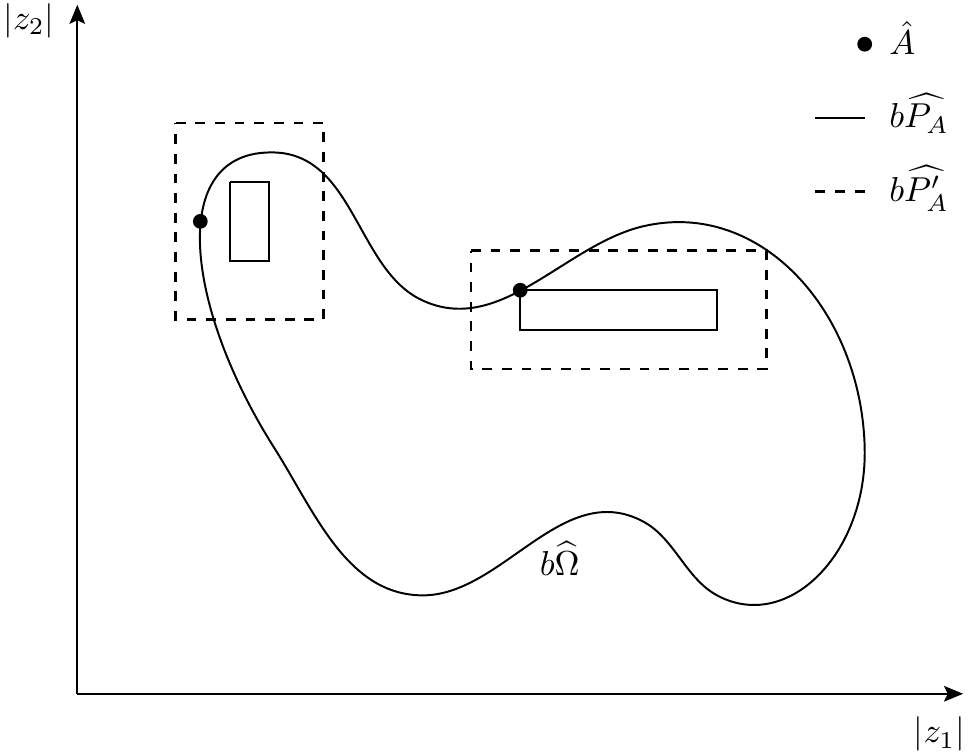}
\caption{Products of Annuli $P_A$ and $P'_A$.}
\end{figure}

We complete the proof by showing \eqref{eqn:ChooseP_z}.
Fix $A=(a_1,\ldots,a_n)\in b\Omega$.
Consider the modulus profile $\widehat{\Omega}=\{\big(\abs{z_1},\ldots,\abs{z_n}\big): (z_1,\ldots,z_n)\in\Omega\}\subset\R^n$.
Now, $\widehat{A}=\big(\abs{a_1},\ldots,\abs{a_n}\big)\in b\widehat{\Omega}$.
Then, there is at least one $1\le j\le n$, such that the $\abs{z_j}$ direction is transversal to $b\widehat{\Omega}$ at $\widehat{A}$.
Without loss of generality suppose that $j=1$ and the increasing $\abs{z_1}$ direction points inward (into $\widehat{\Omega}$) at $\widehat{A}$.
Now there is an $0<\epsilon<\abs{a_1}$ such that
\[P_A = \A\big(\abs{a_1}+\epsilon,\abs{a_1}+3\epsilon\big)\times\prod_{k=2}^n\A\big(\abs{a_k}-\epsilon,\abs{a_k}+\epsilon\big) \subset \Omega.\]
Notice that the inner radius of the first annulus in the product $P_A'$, from Theorem~\ref{thm:MainLocalFull}, is 
\[\frac{\left(\abs{a_1}+\epsilon\right)^2}{\abs{a_1}+3\epsilon} < \abs{a_1}\]
by the choice of $\epsilon$.
Hence, $A\in P_A'$ and we are done.
\end{proof}


\section{Concluding Remarks}\label{sec:ConcRmks}

In light of the results presented here, it is natural to ask whether this holomorphic extension phenomenon holds on other classes of domains.

As pointed out in the introduction, we expect our results to have operator theoretic implications for the Friedrichs operator and in the study of quadrature domains.
For a similar study on domains that satisfy condition R see \cite{PranavKaushal}.

Recently, in \cite{ChaDupMou} it is shown that the smoothing property observed in \cite{HerMcN, HerMcNStr} also holds for weighted Bergman projections.
A direct adaptation of our arguments here would show that the same holomorphic extension also holds for weighted Bergman projections with multi-radial weights. 

A version of the smoothing property shown in \cite{HerMcN} also holds for the harmonic Bergman projection \cite{Her13}.
Analogously, one can investigate harmonic extension phenomena enjoyed by the harmonic Bergman projection on domains with certain symmetries.


\section*{Acknowledgements}

We would like to thank the anonymous referee for constructive comments. We wish to thank Kaushal Verma for pointing out the connection between our work and the Friedrichs operator. Additionally, we thank Mehmet \c{C}elik and Jeffery D. McNeal for their helpful feedback during the course of this project.

\def\MR#1{\relax\ifhmode\unskip\spacefactor3000 \space\fi \href{http://www.ams.org/mathscinet-getitem?mr=#1}{MR#1}}

\end{document}